\documentclass{amsart}
\usepackage{amsmath}
\usepackage{graphicx}
\usepackage{latexsym}
\usepackage{color}
\usepackage{amscd}
\usepackage[all]{xy}
\usepackage{enumerate}
\usepackage{hyperref}
\usepackage{subfigure}
\usepackage{amssymb}
\theoremstyle{plain}
\newtheorem{theorem}{Theorem}[section]

\newtheorem{lemma}[theorem]{Lemma}
\newtheorem{proposition}[theorem]{Proposition}

\theoremstyle{definition}
\newtheorem{definition}[theorem]{Definition}

\newtheorem{remark}[theorem]{Remark}

\theoremstyle{definition}



\DeclareMathOperator{\Map}{Map}
\DeclareMathOperator{\Diff}{Diff}
\DeclareMathOperator{\Sym}{Sym}
\DeclareMathOperator{\h}{\mathcal{H}}
\DeclareMathOperator{\I}{\mathcal{I}}

\newcommand{\CPb}{\overline{\mathbb{CP}}{}^{2}}
\newcommand{\CP}{{\mathbb{CP}}{}^{2}}

\newcommand{\R}{\mathbb{R}}
\newcommand{\Q}{\mathbb{Q}}

\newcommand{\N}{\mathbb{N}}

\newcommand{\Pa}{\partial}
\newcommand{\mc}[1]{\mathcal{#1}}

\def\Ker{\operatorname{Ker}}

\def\Diff{\operatorname{Diff}}

\def\id{\operatorname{id}}

\def\Crit{\operatorname{Crit}}

\def\Sign{\operatorname{Sign}}

\def \x {\times}
\def \eu{{\text{e}}}

\begin{document}

\title[Broken Lefschetz fibrations and mapping class groups]
{Broken Lefschetz fibrations \\ and mapping class groups}

\author[R. \.{I}. Baykur]{R. \.{I}nan\c{c} Baykur}
\address{Department of Mathematics and Statistics, University of Massachusetts, Amherst, MA 01003-9305, USA}
\email{baykur@math.umass.edu}

\author[K. Hayano]{Kenta Hayano}
\address{Department of Mathematics, Graduate School of Science, Hokkaido University, Sapporo, Hokkaido 060-0810, Japan}
\email{k-hayano@math.sci.hokudai.ac.jp}

\begin{abstract}
The purpose of this note is to explain a combinatorial description of closed smooth oriented $4$-manifolds in terms of positive Dehn twist factorizations of surface mapping classes, and further explore these connections. This is obtained via monodromy representations of simplified broken Lefschetz fibrations on $4$-manifolds, for which we provide an extension of Hurwitz moves that allows us to uniquely determine the isomorphism class of a broken Lefschetz fibration.  We furthermore discuss broken Lefschetz fibrations whose monodromies are contained in special subgroups of the mapping class group; namely, the hyperelliptic mapping class group and in the Torelli group, respectively, and present various results on them which extend or contrast with those known to hold for honest Lefschetz fibrations. Lastly, we show that there are $4$-manifolds admitting infinitely many pairwise nonisomorphic relatively minimal broken Lefschetz fibrations with isotopic regular fibers.   
\end{abstract}

\maketitle

\setcounter{secnumdepth}{2}
\setcounter{section}{0}

\section{Introduction} 

Broken Lefschetz fibrations (BLF in short) are smooth surjective maps from compact smooth oriented $4$-manifolds to orientable surfaces, the very definition of which allow certain unstable and stable singularities to coexist: namely, their singular set consists of indefinite fold singularities along embedded circles and Lefschetz type singularities on a discrete set disjoint from the former. Let $X$ and $B$ be compact oriented manifolds of dimension four and two, respectively, and $f\colon\, X\to B$ be a smooth map. Given any surjective continuous map from a closed oriented $4$-manifold $X$ to the $2$-sphere, there exists a rather special BLF on $X$ within the same homotopy class, called a \emph{simplified broken Lefschetz fibration}, satisfying the following properties: it has only connected fibers and no exceptional spheres contained on the fibers, has at most one circle of indefinite fold singularities whose image in the base is embedded, and all of its Lefschetz critical points lie on distinct fibers of the highest genera. These fibrations were introduced by the first author in \cite{B1}, and allow one to study the underlying topology of the $4$-manifold effectively. (See for instance \cite{B1, B2, B3, BK, H1, H2, HS1, HS2}.) 

The underlying topology of a genus-$g$ simplified BLF is rather simple: It is either a relatively minimal genus $g$ Lefschetz fibration over the $2$-sphere, or it decomposes as a relatively minimal genus $g$ Lefschetz fibration over a $2$-disk, a trivial genus $g-1$ bundle over a $2$-disk, and a fibered cobordism in between prescribed by a single round handle \cite{B1}. Unlike general BLFs, this \textit{simplified} picture presents two great advantages: (1) it allows one to equip the $4$-manifold $X$ with a rather simple handlebody decomposition, and in turn makes it possible to identify the total space of the fibration via handle calculus; and (2) it leads to a description of $X$ in terms of curves on surfaces and factorizations of mapping class groups in terms of positive Dehn twists. The current article will elaborate on the latter aspect, in hopes to convince the reader that the study of smooth $4$-manifolds can be effectively translated to that of certain positive factorizations in the mapping class groups. By a slight abuse of language, we will thus talk about BLFs which are always simple, without further mentioning of the extra assumptions.
\footnote{An analogous generalization of fibrations on $4$-manifolds are the \textit{simplified purely wrinkled fibrations} studied by Williams \cite{W}, which are also prescribed in terms of curves on a surface. However, the main disadvantage in the latter is the lack of understanding of when such a description indeed hands us a wrinkled fibration over the $2$-sphere (and not just over the $2$-disk). The only sensible way of overcoming this difficulty has been perturbing them to simplified BLFs and then appealing to mapping class group arguments we will detail here.}

Let $\Map(\Sigma_g)$ denote the mapping class group of orientation-preserving diffeomorphisms of the genus $g$ orientable surface $\Sigma_g$ and assume that $g \geq 3$. Given a genus $g$ (simplified) BLF $f \colon X \to S^2$, there is an associated ordered tuple of cycles $(c; c_1,\ldots,c_n)$, where $c, c_1, \ldots, c_n$ are simple closed curves on $\Sigma_g$. Such a tuple of cycles prescribes a genus-$g$ BLF, if and only if 
\[ \mu= t_{c_n} \cdot \ldots \cdot t_{c_1} (c) = \pm c  \, , \text{and} \]
\[ \mu \in \text{Ker}(\Phi_c: \Map(\Sigma_{g}) \to \Map(\Sigma_{g-1})) \, , \]
as discussed in \cite{B1}. In this case, we call $(c; c_1,\ldots,c_n)$ the \emph{Hurwitz cycle system} $W_f$ associated to $f$ (notice how the ordered tuple has a distinguished first entry). The first part of our paper aims to obtain a one-to-one correspondence between BLFs and Hurwitz cycle systems modulo some natural equivalence relations, similar to the case of Lefschetz fibrations. 

Two BLFs $f_i: X_i\rightarrow B_i$, $i=1,2$ are said to be \emph{isomorphic} if there exist orientation preserving diffeomorphisms $\Phi: X_1\rightarrow X_2$ and $\varphi: B_1\rightarrow B_2$ such that $f_2\circ \Phi = \varphi\circ f_1$. On the other hand, given a Hurwitz cycle system $W=(c; c_1, \ldots, c_n)$, let an \emph{elementary transformation} of $W$ be any modification of the type
\[
(c; c_1,\ldots ,c_i,c_{i+1}, \ldots  , c_n) \longrightarrow (c; c_1,\ldots, c_{i+1}, t_{c_{i+1}}(c_{i}), \ldots, c_n) ,
\]
and for any $h \in \Map(\Sigma_g)$ let the \emph{simultaneous action by $h$} on $W$ be the modification 
\[
(c; c_1,\ldots, c_n) \longrightarrow (h(c); h(c_1),\ldots, h(c_n))
\]
both of which are easily seen to be resulting in new Hurwitz cycle systems. We call two Hurwitz cycle systems $W_1$ and $W_2$ \emph{equivalent} or \emph{Hurwitz equivalent}, if one can be obtained from another via a sequence of elementary transformations, simultaneous actions, and their inverses. 

In Section~\ref{uniqueness}, we introduce the above moves and prove our main result, which gives a complete combinatorial description of (simplified) broken Lefschetz fibrations in terms of Hurwitz cycle systems: 

\begin{theorem} \label{bijectionthm}
For $g \geq 3$ there exists a bijection:
\begin{eqnarray*}
\left\{\begin{array}{c}
\mbox{Genus $g$ }\\
\mbox{broken Lefschetz fibrations}\\
\mbox{up to isomorphism}\\
\end{array}\right\}
&
\longleftrightarrow
&
\left\{\begin{array}{c}
\mbox{Hurwitz cycle systems }\\
\mbox{up to Hurwitz equivalence }\\
\end{array}\right\}
\end{eqnarray*}

\end{theorem}

\noindent Regarding genuine Lefschetz fibrations as BLFs with empty round locus, the above bijection restricts to the well-known bijection between isomorphism classes of Lefschetz fibrations and Hurwitz cycle systems due to Kas \cite{Kas} and Matsumoto \cite{Matsumoto}, independently. 

Recall that the \textit{monodromy} of a genus-$g$ Lefschetz fibration with monodromy representation 
\[ t_{c_n} \cdot \ldots \cdot t_{c_1} =1  \]
is said to be contained in a subgroup $N < \Map(\Sigma_g)$ if all $t_{c_1}, \ldots, t_{c_n}$, possibly after conjugating all with the same mapping class, lie in $N$. There are two special subgroups of $\Map(\Sigma_g)$, the hyperelliptic mapping class group $\h(\Sigma_g)$ and the Torelli group $\I(\Sigma_g)$, which are generated by the elements in $\Map(\Sigma_g)$ commuting with a fixed hyperelliptic involution on $\Sigma_g$, and those which act trivially on the first homology $H_1(\Sigma_g)$, respectively. Lefschetz fibrations (possibly over higher genera surfaces) with monodromies contained in $\h(\Sigma_g)$ or $\I(\Sigma_g)$, shortly called as \textit{hyperelliptic Lefschetz fibrations} and \textit{Torelli Lefschetz fibrations}, respectively, has been of special interest, as we briefly discuss below. The second part of our paper will focus on generalizations of these results to BLFs ---where we again consider the images of $t_{c_1}, \ldots, t_{c_n}$. 

Siebert and Tian proved that hyperelliptic Lefschetz fibrations always arise as (symplectic) branched covers of rational ruled surfaces \cite{ST} (also see \cite{Fuller}), and were able to prove that every genus-$2$ Lefschetz fibration with non-separating fibers is holomorphic building on this fact \cite{ST2}. The following result, obtained by the second author and Sato, is a natural extension of the result of Siebert-Tian and Fuller to BLFs:

\begin{theorem}[\cite{HS1}] \label{hyperellipticthm}
Given a genus $g \geq 3$ hyperelliptic broken Lefschetz fibration \linebreak $f: X \to S^2$ with no separating vanishing cycles, there exists an involution $\eta$ on $X$ so that $X$ is a double branched covering of a rational ruled surface. In the presence of separating vanishing cycles, the same holds for a blow-up of $X$.  
\end{theorem}

\noindent In Section~\ref{hyperellipticBLFs} we will discuss this result in some more detail, along with a couple other topological implications and extensions of earlier results on hyperelliptic Lefschetz fibrations.

As for Torelli monodromies, Smith proved that there are no nontrivial relatively minimal (i.e.~without null-homotopic vanishing cycles $c_i$) genus-$g$ (symplectic) Lefschetz fibrations over the $2$-sphere whose monodromy are contained in $\I(\Sigma_g)$ \cite{Smith}. (On the other hand, the first author and Margalit showed that this is no longer true when the base surface is of higher genus; namely, they constructed infinitely many pairwise inequivalent genus-$g$ nontrivial Torelli Lefschetz fibrations over genus-$h$ surfaces for all $g \geq 3$ and $h \geq 2$.) In contrast, we observe that there are $4$-manifolds admitting relatively minimal (near-symplectic) nontrivial Torelli BLFs, whereas we note that such examples are obstructed in a natural way that extends the result of Smith ---bear in mind that every nontrivial Lefschetz fibration admits an almost complex structure:

\begin{theorem} \label{torellithm}
An almost complex $4$-manifold does not admit a nontrivial relatively minimal Torelli broken Lefschetz fibration. Moreover, for all $g \geq 2$, there are infinitely many genus-$g$ nontrivial relatively minimal Torelli broken Lefschetz fibrations on near-symplectic $4$-manifolds.  
\end{theorem}

Lastly, we turn to the question of the existence of nonisomorphic genus-$g$ fibrations for a fixed $g$ on a given $4$-manifold. Recently many families of nonisomorphic genus-$g$ Lefschetz fibrations and pencils have been produced \cite{ParkYun, B5, BHmultisection}. Nevertheless, there are no known examples of \textit{infinite families} of Lefschetz fibrations on a symplectic $4$-manifold $X$, and it is undetermined whether or not the diffeomorphism class of the fiber uniquely determines a Lefschetz fibration/pencil on $X$ up to isomorphisms. None of the examples above provide pairs of Lefschetz fibrations $f_i: X \to S^2$ with regular fibers $F_i$, for $i=1,2$, such that there is an ambient diffeomorphism of $X$ taking $F_1$ to $F_2$. In Section~\ref{nonisomorphicexamples}, we prove that, in the case of BLFs, there are indeed examples striking all these features:

\begin{theorem} \label{infinite}
For every $g \geq 2$, there are closed oriented $4$-manifolds admitting infinitely many pairwise nonisomorphic relatively minimal genus-$g$ broken Lefschetz fibrations with isotopic regular fibers.   
\end{theorem}

\noindent That is, we have a family of BLFs $(X, f_i)$, $i \in \N$, with regular fibers $F_i$, such that all $F_i$ as framed surfaces are smoothly isotopic in $X$. This demonstrates the rigidity of BLFs as opposed to their equivalence via regular homotopies (see Remark~\ref{cohomotopy}).

\vspace{0.2cm}
\section{Preliminaries} 

\vspace{0.2cm}
\subsection{Broken Lefschetz fibrations} \

Let $X$ and $B$ be smooth, compact, oriented, manifolds of dimension $4$ and $2$, respectively, and $f\colon\, X\to B$ a smooth map. The map $f$ is said to have a \emph{Lefschetz singularity} at a point $x \in Int(X)$, if around $x$ and $f(x)$ one can choose orientation preserving charts so that $f$ conforms the complex local model
\[(u, v) \mapsto u^2 + v^2 .\]
The map $f$ is said to have a \emph{round singularity} (or \emph{indefinite fold singularity}) along a $1$-manifold $Z \subset X$ if around every $z \in Z$, there are coordinates $(t, x_1, x_2, x_3)$ with $t$ a local coordinate on $Z$, in terms of which $f$ is given by
\[(t, x_1, x_2, x_3) \mapsto (t, x_1^2 + x_2^2 - x_3^2).\]
A \emph{broken Lefschetz fibration} is then defined as a surjective smooth map $f\colon\, X\to B$ which is submersion everywhere except for a finite set of points $C$ and a finite collection of disjoint embedded circles $Z \subset X \setminus C$, where it has Lefschetz singularities and round singularities, respectively. We call the $1$-manifold $Z$ the \emph{round locus} and its image $f(Z)$ the \emph{round image} of $f$. These fibrations were first introduced by Auroux, Donaldson and Katzarkov in \cite{ADK}. 

\emph{Simplified broken Lefschetz fibrations} constitute a subfamily of broken Lefschetz fibrations subject to the following additional conditions \cite{B1}: The base surface is $B=S^2$, the round image is connected (and possibly empty), the round image is embedded, the fibration is relatively minimal, and whenever $Z \neq \emptyset$, all the fibers are connected and all the Lefschetz singularities lie over the $2$-disk component of $S^2 \setminus f(Z)$ over which fibers have higher genera. Note that a (simplified) broken Lefschetz fibration with $Z = \emptyset$ is an honest Lefschetz fibration, thus they are natural generalizations of Lefschetz fibrations over the $2$-sphere. Importantly, in any homotopy class of a map from $X$ to $S^2$ there exists a representative which is a simplified BLF \cite{B4, W}, which demonstrates that these fibrations are found in abundance. As we have mentioned in the Introduction, with no further notice, we will assume hereon that all the broken Lefschetz fibrations we work with are simplified, and still call them BLFs.

\vspace{0.2cm}
\subsection{Mapping class groups}\label{notation_MCG} \

Let $\Sigma$ be a compact, oriented and connected surface, $c\subset \Sigma$ a simple closed curve and $V_i$, $i=1, \ldots, n$, be a finite collection of points on $\Sigma$.  We define a group $\Map{(\Sigma; V_1,\ldots,V_n)}(c)$ as follows: 
\[\Map{(\Sigma; V_1,\ldots, V_n)}(c) = \left\{[T]\in \pi_0(\Diff^+{(\Sigma; V_1,\ldots, V_n)}) \hspace{.3em} | \hspace{.3em} T(c)=c\right\}, 
\]
where $\Diff^+{(\Sigma;V_1,\ldots, V_n)}$ consists of orientation-preserving diffeomorphisms of $\Sigma$ fixing $\Pa \Sigma$ pointwise and each $V_i$ setwise. 
In this paper, we define a group structure on the above group by the composition as maps, that is, for elements $T_1,T_2\in\Diff^+(\Sigma; V_1,\ldots, V_n)$, we define the product $T_1\cdot T_2$ as follows: 
	\begin{equation*}
	T_1\cdot T_2 = T_1\circ T_2. 
	\end{equation*}
We define a group structure of $\Map(\Sigma; V_1, \ldots, V_n)(c)$ in the same way.

\vspace{0.2in}
\section{Monodromy representations of broken Lefschetz fibrations} \label{uniqueness}

For any genus-$g$ Lefschetz fibration $f : X \to S^2$ we obtain a \textit{monodromy representation} (or \textit{monodromy factorization} 
\[ t_{c_n} \cdot \ldots \cdot t_{c_1} = 1 , \]
in $\Map(\Sigma_g)$. Such a factorization can also be encoded by an ordered tuple of curves $(c_1,\ldots,c_n)$ on $\Sigma_g$, called the Hurwitz system $W_f$ for $f$.  A classical result of Earle and Eells states that the connected components of the diffeomorphism group $\Diff(\Sigma_g)$ is contractible \cite{EE}. So, for $g \geq 2$, one can recover a genus-$g$ Lefschetz fibration from a factorization of the identity into positive Dehn twists in $\Map(\Sigma_g)$. By the works of Kas and Matsumoto \cite{Kas, Matsumoto}, one then obtains a one-to-one correspondence between genus $g \geq 2$ Lefschetz fibrations, up to isomorphisms, and monodromy factorizations, up to global conjugation of all elements by mapping classes in $\Map(\Sigma_g)$ and Hurwitz moves ---as we will review below.  The purpose of this section is to extend this very useful combinatorial correspondence to BLFs.

\vspace{0.2cm}
\subsection{Round cobordism and capping homomorphism} \

The round cobordism, containing the $1$-dimensional indefinite fold, provides a \textit{fibered cobordism} between a $\Sigma_g$-bundle and a $\Sigma_{g-1}$-bundle over $S^1$. This is essentially what differentiates a BLF from a genuine Lefschetz fibration, and thus deserves a careful study within our translation to mapping class groups. The first ingredient we need here is the homomorphism $\Phi_c$ which will relate the mapping class group of $\Sigma_g$ with that of $\Sigma_{g-1}$.

Let $c\subset \Sigma$ be a non-separating simple closed curve. For a given element \linebreak $\psi\in \Map{(\Sigma)}(c)$, we take a representative $T:\Sigma\rightarrow \Sigma\in \Diff^+{(\Sigma)}$ preserving the curve $c$ setwise.
The restriction $T|_{\Sigma\setminus c}: \Sigma\setminus c\rightarrow \Sigma\setminus c$ is also a diffeomorphism. 
Let $\Sigma_c$ be the surface obtained by attaching two disks to $\Sigma\setminus c$ along $c$. 
The diffeomorphism $T|_{\Sigma\setminus c}$ can be extended to a diffeomorphism $\tilde{T}: \Sigma_c\rightarrow \Sigma_c$. 
We define $\widetilde{\Phi_c}([T])\in \Map{(\Sigma_c; \{p_1,p_2\})}$ as the isotopy class of $\tilde{T}$, where $p_1,p_2$ are the origins of the attached disks. 

\begin{lemma}[Also see \cite{Behrens}]\label{lem_welldefinedness}

The map $\widetilde{\Phi_c} : \Map(\Sigma)(c) \rightarrow \Map(\Sigma_c; \{p_1, p_2\})$ is a well-defined, surjective homomorphism, and the kernel of this map is an infinite cyclic group generated by $t_{c}$. 

\end{lemma}

\begin{proof}

We denote by $\Map(\Sigma)(c^{\text{ori.}})$ be a subgroup of $\Map(\Sigma)(c)$ whose element is represented by a map preserving an orientation of $c$. 
There exists the following exact sequence: 
\[
1 \rightarrow \Map(\Sigma)(c^\text{ori.}) \hookrightarrow \Map(\Sigma)(c) \xrightarrow{\varepsilon} \mathbb{Z}/2\mathbb{Z} \rightarrow 1, 
\]
where $\varepsilon : \Map(\Sigma)(c)\rightarrow \mathbb{Z}/2\mathbb{Z}$ is a homomorphism defined as follows: 
\begin{equation*}
\varepsilon(\varphi) = \begin{cases}
0 & \text{(if $\varphi$ is represented by a map preserving an orientation of $c$)}, \\
1 & \text{(otherwise)}.
\end{cases} 
\end{equation*} 
Furthermore, this sequence is split by a hyperelliptic involution $\iota$ of $\Sigma$ preserving the curve $c$. 
In particular, we can obtain the following isomorphism: 
\begin{equation*}\label{E:isomMCG1}
\Theta_1: \Map(\Sigma)(c) \to \Map(\Sigma)(c^\text{ori.}) \rtimes \mathbb{Z}/2\mathbb{Z}. 
\end{equation*}
There also exists the following exact sequence: 
\[
1 \rightarrow \Map(\Sigma_c; p_1,p_2) \hookrightarrow \Map(\Sigma_c; \{p_1,p_2\}) \xrightarrow{\varepsilon^\prime} \mathbb{Z}/2\mathbb{Z}\rightarrow 1, 
\]
where the value $\varepsilon^\prime(\varphi)$ is $0$ if $\varphi$ is contained in $\Map(\Sigma_c, p_1,p_2)$ and $1$ otherwise. 
We take a small regular neighborhood $\nu c$ of $c$ in $\Sigma$. 
This sequence is also split by $\iota^\prime$, which is an involution of $S_c$ induced by $\iota$. 
Thus we obtain the following isomorphism:  
\begin{equation*}\label{E:isomMCG2}
\Theta_2:\Map(\Sigma_c; \{p_1,p_2\}) \to \Map(\Sigma_c; p_1,p_2) \rtimes \mathbb{Z}/2\mathbb{Z}. 
\end{equation*}

Let $\eta: \Map(\Sigma\setminus \nu c) \rightarrow \Map(\Sigma)(c^\text{ori.})$ be a homomorphism induced by the inclusion. 
This map is surjective since every element in $\Map(\Sigma)(c^\text{ori.})$ is represented by a map preserving $\nu c$ pointwise. 
The kernel of $\eta$ is an infinite cyclic group generated by $t_{\delta_1} \cdot {t_{\delta_2}}^{-1}$, where $\delta_1, \delta_2\subset \Sigma\setminus \nu c$ are simple closed curves parallel to the boundary components (see \cite[Theorem 3.18]{FM}).  
We denote by $Cap: \Map(\Sigma\setminus \nu c) \rightarrow \Map(\Sigma_c; p_1,p_2)$ the capping homomorphism. 
This map is surjective and the kernel of it is generated by $t_{\delta_1}, t_{\delta_2}$ (cf.~\cite[Proposition 3.19]{FM}). 
Thus we can take a homomorphism $\widetilde{\Phi_c}^\text{ori.}: \Map(\Sigma)(c^\text{ori.}) \rightarrow \Map(\Sigma_c; p_1,p_2)$ which makes the following diagram commute: 
\begin{equation*}
\xymatrix{
\Map(\Sigma\setminus \nu c) \ar[r]^\eta \ar[d]_{Cap} & \Map(\Sigma)(c^\text{ori.}) \ar[dl]^{\widetilde{\Phi_c}^\text{ori.}} \\
\Map(\Sigma_c; p_1,p_2) & . 
}
\end{equation*}
Moreover, it is easy to see that the kernel $\Ker(\widetilde{\Phi_c}^\text{ori.})$ is an infinite cyclic group generated by $t_c$.

It follows from the definitions that $\widetilde{\Phi_c} = \Theta_2^{-1}\circ (\widetilde{\Phi_c}^\text{ori.} \times \text{id})\circ \Theta_1$.  
The map $\widetilde{\Phi_c}$ is surjective since the map $\widetilde{\Phi_c}^\text{ori.}$ is surjective. 
The kernel of $\widetilde{\Phi_c}$ is isomorphic to that of $\widetilde{\Phi_c}^\text{ori.}$ (via $\Theta_1$), which is generated by $t_c$. 
This completes the proof of Lemma \ref{lem_welldefinedness}. 
\end{proof}

Let $F_{p_1, p_2} : \Map(\Sigma_c; \{p_1,p_2\}) \rightarrow \Map(\Sigma_c)$ be the forgetful homomorphism. 
We define a homomorphism $\Phi_c$ as follows: 
\[
\Phi_c := F_{p_1,p_2} \circ \widetilde{\Phi_c} : \Map(\Sigma)(c) \rightarrow \Map(\Sigma_c). 
\]
It is well-known that the kernel $\Ker(F_{p_1,p_2})$ is generated by point pushing maps along elements in $\pi_1(\Sigma_c\setminus \{p_i\}, p_j)$ ($i\neq j$) and disk twists along arcs between $p_1$ and $p_2$ (See \cite{FM}, for example). 
Note that a lift of such a disk twist under $\widetilde{\Phi_c}$ is equal to a {\it $\Delta$-twist} $\Delta_{c, \beta} = (t_c\cdot t_{\beta})^3$ up to $t_c$, where $\beta$ is a simple closed curve in $\Sigma$ which intersects $c$ at a unique point transversely (the name of $\Delta$-twist was first used in \cite{Behrens}). 
Combining this observation with Lemma \ref{lem_welldefinedness}, we obtain:  

\begin{lemma}\label{lem_generator_kernel}

The group $\Ker(\Phi_c)$ is generated by lifts of point pushing maps along $\pi_1(\Sigma_c\setminus \{p_i\}, p_j)$ ($i\neq j$) under $\widetilde{\Phi_c}$, a $\Delta$-twist $\Delta_{c, \beta}$, and the Dehn twist $t_c$. 

\end{lemma}

\vspace{0.2cm}
\subsection{Monodromy representation and Hurwitz equivalence for BLFs} \

We assume that $B=D^2$ and we put $f(\mathcal{C}_f)=\{y_1,\ldots,y_n\}$. 
We take embedded paths $\alpha_1,\ldots,\alpha_n$ in $D^2$ satisfying the following conditions: 

\begin{itemize}

\item each $\alpha_i$ connects $y_0$ to $y_i$, 

\item if $i\neq j$, then $\alpha_i\cap\alpha_j=\{y_0\}$, 

\item $\alpha_1,\ldots,\alpha_n$ appear in this order when we travel counterclockwise around $y_0$. 

\end{itemize}

\noindent
The paths $\alpha_1,\ldots, \alpha_n$ give vanishing cycles $c_1,\ldots,c_n \subset \Sigma_g$ of Lefschetz singularities of $f$. 

Let $f:X\rightarrow S^2$ be a BLF with $Z_f\neq \emptyset$, $X_h$ be the higher side of $f$, and $X_r$ be the round cobordism of $f$. 
The restriction $f|_{X_h}$ is an LF over $D^2$. 
As explained above, we can obtain vanishing cycles of $c_1,\ldots,c_n$ by taking paths $\alpha_1,\ldots, \alpha_n$ for $f|_{X_h}$. 
We further take a path $\alpha\subset f(X_h\cup X_r)$ from $y_0$ to a point on the image of indefinite folds of $f$ so that $\alpha$ intersects each of other paths $\alpha_1,\ldots, \alpha_n$ only at $y_0$, and that $\alpha, \alpha_1, \ldots, \alpha_n$ appear in this order when we go around $y_0$ counterclockwise. 
We call a system of paths $\alpha, \alpha_1,\ldots, \alpha_n$ satisfying the condition above a {\it Hurwitz path system} of $f$. 
The path $\alpha$ gives a vanishing cycle $c\subset \Sigma_g$ of indefinite folds.  
We denote by $W_f$ a sequence of vanishing cycles $(c; c_1,\ldots,c_n)$. 

\begin{lemma}[Also see \cite{B1}]\label{lem_relation_BLF_MCG}

Let $f:X\rightarrow S^2$ be a genus-$g$ broken Lefschetz fibration, and $W_f= (c; c_1,\ldots,c_n)$ be a sequence of vanishing cycles of $f$ obtained as above. 
The product $t_{c_n}\cdot \cdots \cdot t_{c_1}$ is contained in the kernel $\Ker(\Phi_c)$, where $t_{c_i}$ is the right-handed Dehn twist along $c_i$. 

\end{lemma}

To give a proof of Lemma~\ref{lem_relation_BLF_MCG} we need the following: 

\begin{lemma}\label{T:model_fold}

Let $Z\subset X$ be the set of indefinite folds of $f$. 
There exist diffeomorphisms $\Psi: I \times B^1 \times B^2/(1,x,y_1,y_2)\sim (0, \pm x,y_1,\pm y_2) \rightarrow \nu(Z)$ and \linebreak $\psi:I\times B^1/(1,X)\sim (0,X) \to \nu f(Z)$ such that the following diagram commutes: 
\[
\begin{CD}
I \times B^1 \times B^2/\sim @> \Psi >> \nu(Z)\\
@V \mu VV @VV f V \\
I \times B^1/ \sim @>\psi>> \nu f(Z),  
\end{CD}
\]
where $B^i$ is an open ball of dimension $i$ and $\mu$ is defined by 
\[ \mu (t,x,y_1,y_2) = (t, -x^2+ {y_1}^2+{y_2}^2) \, . \] 
\end{lemma}

\begin{proof}[Proof of Lemma~\ref{T:model_fold} (Sketch)]
First take a diffeomorphism $\psi: (I\times \R)/\sim \, \to \nu f(Z)$. 
We denote by $F$ the composition $\pi\circ \psi^{-1}\circ f$, where $\pi$ is the projection onto $I/\sim$. 
It is easy to see that $F$ is a locally trivial fibration. 
We denote the restriction $f|_{F^{-1}(t)}$ by $f_t$. 
We take a coordinate neighborhood $(U,\varphi)$ of $p_0\in Z\cap F^{-1}(0)$ (and retake a diffeomorphism $\psi$ if necessary) so that 
\begin{itemize}

\item for any $t$, the set $\varphi^{-1}(\{(t,x,y)~|~ (x,y)\in \R^3\})$ is in $F^{-1}(t)$, and

\item $f\circ \varphi^{-1}(t,x,y_1,y_2) = (t,-x^2 + y_1^2 + y_2^2)$ on $\varphi^{-1}(U)$. 

\end{itemize}
For any $t\in I$, we take vector fields $V^-(t), V^+_i(t)$ on $F^{-1}(t)\cap \nu Z$ so that
\begin{itemize}

\item $V^-(t)$ and $V^+_i(t)$ depend smoothly on $t$, 

\item $\{V^-(t)_p, V^+_1(t)_p, V^+_2(t)_p\}$ spans $\Ker(d(f_t)_p)$, where $p\in Z\cap F^{-1}(t)$, 

\item $\delta(V^-(t)_p,V^-(t)_p)$ is negative, while $\delta(V^+_i(t)_p,V^+_i(t)_p)$ is positive, where $\delta: \Sym(\Ker d(f_t)_p)\to \R\left(\frac{\Pa}{\Pa X}\right)$ is the intrinsic derivative of $f_t$, 

\item $V^-(\epsilon)$ and $V^+_i(\epsilon)$ coincide with $d\varphi \left(\frac{\Pa}{\Pa x}\right)$ and $d\varphi \left(\frac{\Pa}{\Pa y_i}\right)$ for sufficiently small $\epsilon$, respectively, and  

\item $V^-(1-\epsilon)$, $V^+_1(1-\epsilon)$ and $V^+_2(1-\epsilon)$ coincide with $\pm d\varphi \left(\frac{\Pa}{\Pa x}\right)$, $d\varphi \left(\frac{\Pa}{\Pa y_1}\right)$ and $\pm d\varphi \left(\frac{\Pa}{\Pa y_1}\right)$ for sufficiently small $\epsilon$
\footnote{The signs here are positive if a line subbundle of $\Ker(dF)|_{Z}$ consisting of subspaces on which $\delta$ is negative definite is orientable and negative otherwise.}
. 

\end{itemize}
Making the vector fields above sufficiently small, we define a map 
\[\Psi: I \times B^1 \times B^2/(1,x,y_1,y_2)\sim (0, \pm x,y_1,\pm y_2) \rightarrow \nu(Z) \] 
as follows: 
\[
\Psi(t,x,y_1,y_2) = C_{x V^-(t)+ y_1V^+_1(t) + y_2V^+_2(t)}(1), 
\]
where $C_{x V^-(t)+ y_1V^+_1(t) + y_2V^+_2(t)}$ is an integral curve of $x V^-(t)+ y_1V^+_1(t) + y_2V^+_2(t)$ with initial point in $Z\cap F^{-1}(t)$. 
It is easy to see that this map is a diffeomorphism provided that $V^-$ and $V^+_i$ are sufficiently small. 
A maximal compact subgroup of the stabilizer of the $3$-dimensional indefinite Morse function germ in $\mathcal{R}$ is isomorphic to $\{(A,B)\in \mathrm{O}(2)\times \mathrm{O}(1)~|~ \det(A)\det(B) =1\}$. 
Thus, using the algorithm in the proof of the Morse Lemma and \cite[Theorem 3]{Janich}, we can obtain the desired diffeomorphisms.  
\end{proof}

\begin{proof}[Proof of Lemma~\ref{lem_relation_BLF_MCG}]
Let $\Psi: I \times B^1 \times B^2/\sim \rightarrow \nu(Z)$ and $\psi:I\times B^1/\sim \to \nu f(Z)$ be diffeomorphisms which satisfy the condition in Lemma~\ref{T:model_fold}. 
Assume that the path $\alpha$ in a Hurwitz path system giving the cycles $c$ coincides with $\psi(\{0\}\times [0,1])$ is the image of $\psi$. 
We take a horizontal distribution $\mathtt{H}$ of the restriction $f|_{X\setminus \Crit(f)}$ so that it is equal to $d\Psi(\mathtt{H}_0)$ on the image $\Psi(I \times B^1_{1/2}\times B^2_{1/2}/\sim)$, where $\mathtt{H}_0$ is the orthogonal complement of $\Ker(d\mu)$ with respect to the Euclidean metric and $B^i_{1/2}$ is an open ball with radius $\frac{1}{2}$. 
Since $t_{c_n}\cdot \cdots \cdot t_{c_1}$ is the global monodromy of the higher side of $f$, this product is represented by the parallel transport of $\mathtt{H}$ along $\psi(I\times \left\{\varepsilon\right\}/\sim)$ for small $\varepsilon >0$, in particular it is contained in $\Map(\Sigma_g)(c)$. 
Furthermore, it is easy to see that the image $\Phi_c(t_{c_n}\cdot \cdots \cdot t_{c_1})$ is represented by the parallel transport of $\mathtt{H}$ along $\psi(I\times \left\{-\varepsilon\right\}/\sim)$, which is the identity since the preimage $f^{-1}(\psi(I\times \left\{-\varepsilon\right\}/\sim))$ bounds the lower side. 
\end{proof}

\begin{definition}

A system $(c; c_1,\ldots, c_n)$ of simple closed curves is called a {\it Hurwitz cycle system} if it satisfies the condition $t_{c_n}\cdot \cdots \cdot t_{c_1}\in \Ker(\Phi_c)$. 
By Lemma \ref{lem_relation_BLF_MCG}, the system $W_f$ derived from a BLF $f$, together with a Hurwitz path system of it, is a Hurwitz cycle system. 
We call this system a {\it Hurwitz cycle system of $f$}. 

\end{definition}

There are two types of modifications of Hurwitz cycle systems. 
The first one, which we will refer to as an {\it elementary transformation}, is as follows: 
\[
(c; c_1,\ldots ,c_i,c_{i+1}, \ldots  , c_n) \longrightarrow (c; c_1,\ldots, c_{i+1}, t_{c_{i+1}}(c_{i}), \ldots, c_n). 
\]
It is easy to see that this modification can be realized by replacing a Hurwitz path system as described in the left side of Figure~\ref{F:Huwitzpath1}. 
The second modification, {\it simultaneous action by $h\in \Map(\Sigma_g)$}, is as follows. 
\[
(c; c_1,\ldots, c_n) \longrightarrow (h(c); h(c_1),\ldots, h(c_n)). 
\]
This modification corresponds to substitution of an identification of the reference fiber with $\Sigma_g$. 

\begin{figure}[htbp]
\centering
\subfigure[]{\includegraphics[height=22mm]{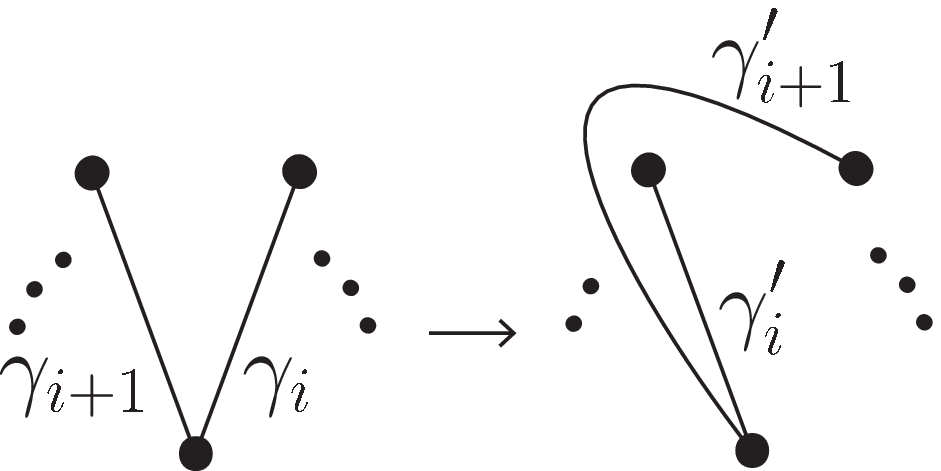}
\label{F:Huwitzpath1}}
\hspace{1em}
\subfigure[]{\includegraphics[height=22mm]{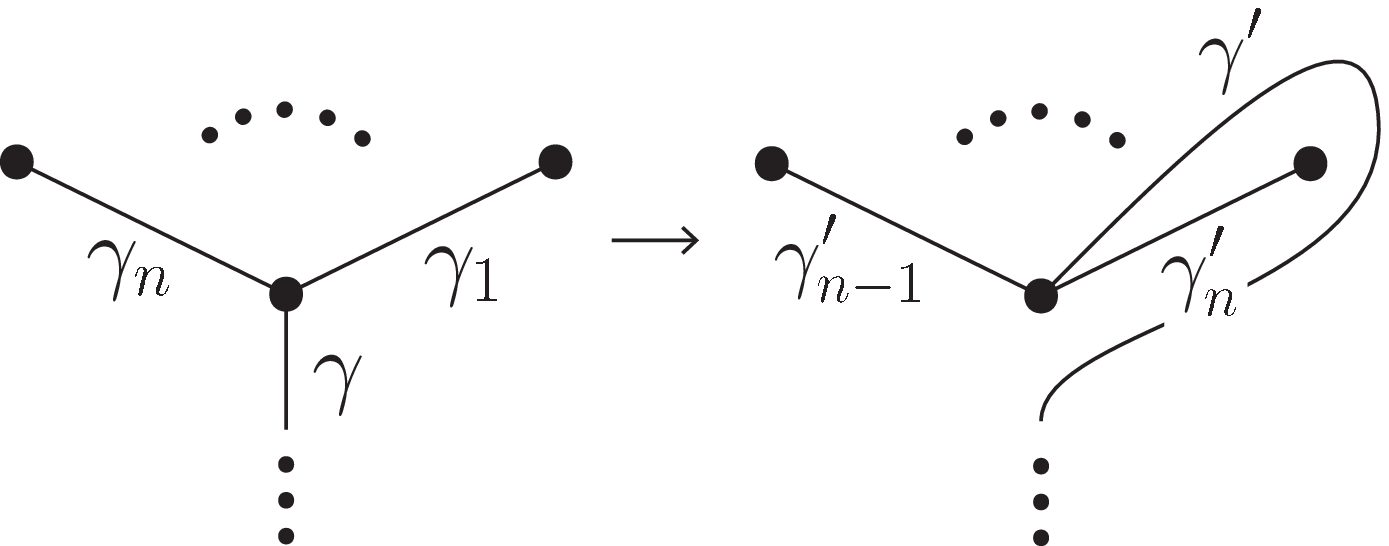}
\label{F:Huwitzpath2}}
\caption{Modifications of Hurwitz path systems. }
\label{modification_path_system}
\end{figure}

\begin{definition}
Two Hurwitz cycle systems are said to be {\it equivalent} if one can be obtained from the other by successive application of elementary transformations, simultaneous actions, and their inverse. 
\end{definition}

\begin{remark}

There is another modification of a Hurwitz cycle system described as follows: 
\[
(c; c_1,\ldots, c_n) \longrightarrow (t_{c_1}(c); c_2,\ldots, c_n, c_1). 
\]
It is easy to verify that this modification is induced by the modification of a Hurwitz path system described in the right side of Figure~\ref{F:Huwitzpath2}. 
Furthermore, this modification can be realized by simultaneous action by $t_{c_1}$, followed by successive application of inverse of elementary transformations. 
This modification will play a key role in the proof of the theorem below. 

\end{remark}

\vspace{0.2cm}
\subsection{Uniqueness of BLFs up to Hurwitz equivalence} \

The following equivalence relation for BLFs naturally extends the notion of isomorphisms of Lefschetz fibrations:
\begin{definition}
Two fibrations $f_1: X_1\rightarrow B_1$ and $f_2: X_2\rightarrow B_2$ are said to be {\it isomorphic} if there exist diffeomorphisms $\Phi: X_1\rightarrow X_2$ and $\varphi: B_1\rightarrow B_2$ satisfying the condition $f_2\circ \Phi = \varphi\circ f_1$. 
\end{definition}

\begin{theorem}\label{thm_isom_equiv}

Let $f_i: X_i \rightarrow S^2$ be a BLF with genus $g\geq 3$ ($i=1,2$). 
The fibrations $f_1$ and $f_2$ are isomorphic if and only if the corresponding Hurwitz cycle systems $W_{f_1}$ and $W_{f_2}$ are equivalent. 

\end{theorem}

\begin{remark}
The statement o the theorem does not hold if the $g \geq 3$ assumption is dropped. There exist infinitely many genus-$1$ BLFs wit the same Hurwitz cycle systems but with pairwise homotopy inequivalent total spaces \cite{BK, H1}.
\end{remark}
We can now prove our first main theorem:
\begin{proof}[Proof of Theorem \ref{thm_isom_equiv}]

We first prove the ``only if'' part. 
Suppose that $f_1$ and $f_2$ are isomorphic, and we fix diffeomorphisms $\Phi: X_1\rightarrow X_2$ and $\varphi: S^2\rightarrow S^2$ satisfying the condition in the definition. 
We take a Hurwitz path system $\gamma, \gamma_1,\ldots, \gamma_n$ of the fibration $f_1$. 
We denote by $W_{f_1}$ the corresponding Hurwitz cycle system of $f_1$ derived from this system, together with an identification $\phi: f_1^{-1}(y_0) \rightarrow \Sigma_g$. 
The system $\varphi(\gamma), \varphi(\gamma_1),\ldots, \varphi(\gamma_n)$ is a Hurwitz path system of $f_2$, and the map $\phi\circ \Phi^{-1}:f_2^{-1}(\varphi(y_0))\rightarrow \Sigma_g$ gives an identification. 
We can obtain a Hurwitz cycle system $W_{f_2}$ of the fibration $f_2$ from them. 
It is easy to verify that $W_{f_1}$ is equal to $W_{f_2}$. 
Thus, all we need to prove is a Hurwitz cycle system of $f_1$ derived from a different Hurwitz path system $\gamma^\prime, \gamma_1^\prime, \ldots, \gamma_n^\prime$ is equivalent to $W_{f_1}$. By a similar argument in \cite{GS} (the solution of Exercise 8.2.7(c)), we see that the system $\gamma^\prime, \gamma_1^\prime, \ldots, \gamma_n^\prime$ can be changed into the system $\gamma, \gamma_1,\ldots, \gamma_n$ up to isotopy by successive application of the two moves in Figure \ref{modification_path_system}. This completes the proof of the ``only if'' part. 

We next prove the ``if'' part. 
By the assumption, we can take Hurwitz path systems of $f_1$ and $f_2$, and identifications of reference fibers with the surface $\Sigma_g$ so that the corresponding Hurwitz cycle systems $W_{f_1}$ and $W_{f_2}$ coincide. 
We decompose $X_i$ into the three parts $X_i^{(h)}, X_i^{(r)}$ and $X_i^{(l)}$, that is, the higher side, the round cobordism and the lower side of $f_i$. 
The restriction $f_i|_{X_i^{(h)}}$ is a Lefschetz fibration. 
By \cite[Theorem 2.4]{Matsumoto}, the fibrations $f_1|_{X_1^{(h)}}$ and $f_2|_{X_2^{(h)}}$ are isomorphic. 
In particular, we can take a fiber-preserving diffeomorphism $\Phi_h: X_1^{(h)} \rightarrow X_2^{(h)}$. 
Since there are no singularities on the boundary of the higher side, we can take a diffeomorphism $\Theta_1:\partial X_1^{(h)}\to T(\varphi)$, where $\varphi: \Sigma_g \rightarrow \Sigma_g$ is a diffeomorphism and $T(\varphi)$ is the mapping torus $I\times \Sigma_g / (1, x) \sim (0, \varphi(x))$. 
We denote the composition $\Theta_1\circ \Phi_h\circ \Theta_1^{-1}$ by $\Theta_2$.

Let $c\subset \Sigma_g$ be the vanishing cycle of indefinite folds in $W_{f_1}$ (note that $c$ is also in $W_{f_2}$). 
By Lemma \ref{lem_relation_BLF_MCG} the isotopy class of $\varphi$ is contained in $\Map(\Sigma_g)(c)$. 
In particular we can assume that $\varphi$ preserves a regular neighborhood $\nu(c)$. 
We denote the set of indefinite folds of $f_i$ by $Z_i \subset X_i$. 
For each $i=1,2$ we take a diffeomorphism $\psi_i:f_i(X_i^{(r)})\to I \times D^1 / (1, t)\sim (0, t)$ so that the composition $f_i\circ \Theta_i^{-1}: T(\varphi) \rightarrow (I \times \{1\})/\sim$ becomes the projection, and that the image $\psi_i(Z_i)$ is the circle $I\times \{0\}/\sim$. 
In the same way as in the proof of Lemma~\ref{T:model_fold} we can take a diffeomorphism 
$\Psi_i: \nu(Z_i) \to I \times D_\varepsilon^1 \times D_\varepsilon^2/(1,x,y_1,y_2)\sim (0, \pm x,y_1,\pm y_2)$
which makes the following diagram commute: 
\[
\begin{CD}
\nu(Z_i) @> \Psi_i >> I \times D^1_\varepsilon \times D^2_\varepsilon/\sim\\
@V f_i VV @VV \mu V \\
f_i(X_i^{(r)}) @>\psi_i>> I \times D^1/\sim,  
\end{CD}
\]
where $D^i_\varepsilon$ is the $i$-dimensional disk with radius $\varepsilon$ sufficiently small and $\mu$ is the map in Lemma~\ref{T:model_fold}. 
For a positive number $s \leq 2$, we define a path $\gamma_{t,s}:[0,s] \rightarrow I\times D^1/\sim$ as $\gamma_{t,s}(x) = (t,1-x)$. 
A connected component of the set $Sub(S^1, \Sigma_g)$ of circles in $\Sigma_g$ is simply connected if $g\geq 2$ (see \cite[Theorem 2.7.H]{Ivanov_2002_MCG} for example).
Thus, we can take a horizontal distribution $\mathtt{H}_i$ of $f_i|_{X_i\setminus Z_i}$ so that it satisfies the following conditions: 
\begin{enumerate}

\item $\mathtt{H}_i$ is equal to the orthogonal complement of $\Ker(d\mu)$ with respect to the Euclidean metric in $\Psi_i^{-1}(I\times D_{2\varepsilon/3}^1\times D_{2\varepsilon/3}^2/\sim)$, 

\item the composition $\Phi_1\circ PT^{\mathtt{H}_1}_{\gamma_{t,1-\frac{\varepsilon^2}{4}}}\circ \Theta_1^{-1}|_{\{t\}\times \nu c}$ sends $\{t\} \times \nu(c)$ to the set 
\[
\pi^{-1}\left(t,\frac{\varepsilon^2}{4}\right) \cap \left\{(t,x,y_1,y_2) \in I\times D_{2\varepsilon/3}^1\times D_{2\varepsilon/3}^2/\sim ~\left|~ |x| \leq \frac{\varepsilon}{2}\right.\right\}, 
\]
where $PT^{\mathtt{H}_i}_{\gamma_{t,1-\frac{\varepsilon^2}{4}}}$ is the parallel transport of $\mathtt{H}_i$ along $\gamma_{t, 1-\frac{\varepsilon^2}{4}}$, and

\item $\Phi_1\circ PT^{\mathtt{H}_1}_{\gamma_{t,1-\frac{\varepsilon^2}{4}}}\circ \Theta_1^{-1}|_{\{t\}\times \nu c}$ is equal to $\Phi_2\circ PT^{\mathtt{H}_2}_{\gamma_{t,1-\frac{\varepsilon^2}{4}}}\circ \Theta_2^{-1}|_{\{t\}\times \nu c}$. 

\end{enumerate}

\noindent
We can define a diffeomorphism $\widehat{\Phi}_r: f_1^{-1}\psi_1^{-1}(I\times (-\varepsilon/3,1]) \rightarrow f_2^{-1}\psi_2^{-1}(I\times (-\varepsilon/3,1])$ as follows: 
\begin{equation*}
\widehat{\Phi}_r(w) = \begin{cases}
\Psi_2^{-1}(t,x,y_1,y_2)  & (w = \Psi_1^{-1}(t,x,y_1,y_2) \in \nu(Z_1)\cap I \times D_{2{\varepsilon}/{3}}^1\times D_{2{\varepsilon}/{3}}^2/\sim), \\
PT^{\mathtt{H}_2}_{\gamma_{t,s}}\Theta_2^{-1}(z) & (w= PT^{\mathtt{H}_1}_{\gamma_{t,s}}\Theta_1^{-1}(z)\in X_1^{(r)}). 
\end{cases}
\end{equation*}
It is easy to see that this map is well-defined and preserves fibers. 
Furthermore, we can extend $\widehat{\Phi}_r$ to a fiber-preserving diffeomorphism from $X_1^{(r)}$ to $X_2^{(r)}$. 
We denote the resulting diffeomorphism by $\Phi_r$. 

The restriction $\Phi_r: \partial X_1^{(l)} \rightarrow \partial X_2^{(l)}$ is a fiber-preserving diffeomorphism. 
Since the connected component of the group $\Diff^+(\Sigma_{g-1})$ is contractible if $g\geq 3$, this restriction can be extended to a fiber-preserving diffeomorphism $\Phi_l: X_l^{(1)}\rightarrow X_l^{(2)}$. 
Combing the three diffeomorphisms $\Phi_h, \Phi_r$ and $\Phi_l$, we can obtain the desired map $\Phi$. 
This completes the proof of Theorem \ref{thm_isom_equiv}. 
\end{proof}

\vspace{0.2in}
\section{Broken Lefschetz fibrations with special monodromies} \label{special}

The monodromy $(c; c_1,\ldots, c_n)$ of a BLF is said to be \textit{contained} in the subgroup $N < \Map(\Sigma_g)$, if we can have all the Dehn twists $t_{c_i}$ lie in $N$ after conjugating with the same element in $\Map(\Sigma_g)$. Here we study two subfamilies of BLFs, whose monodromies are contained in the hyperelliptic mapping class group $\h(\Sigma_g)$ and the Torelli group $\I(\Sigma_g)$, respectively. We will then turn back to isomorphisms of BLFs, and for all $g \geq 2$, we will produce an infinite family of examples of nonisomorphic genus-$g$ BLFs with the same regular fiber in the ambient $4$-manifold.

\vspace{0.2cm}
\subsection{Hyperelliptic broken Lefschetz fibrations} \label{hyperellipticBLFs} \

As demonstrated by the work of the second author with Sato in \cite{H1, H2}, main topological results for hyperelliptic Lefschetz fibrations can be extended to BLFs. Let  $f: X \to S^2$ be a genus-$g$ BLF with Hurwitz cycle system $W_f=(c; c_1,\ldots, c_n)$. We say that $f$ is a \textit{hyperelliptic broken Lefschetz fibration} if, possibly after conjugating the curves of $W_f$ by a mapping class in $\Map(\Sigma_g)$, we can assume that the monodromy is contained in $\h(\Sigma_g)$ and $c$ is fixed by the hyperelliptic involution.\footnote{Since $t_c$ commutes with the hyperelliptic involution if and only if $c$ is a symmetric curve with respect to the involution, in this case, one can simply ask all $c, c_1, \ldots, c_n$ to be fixed by a hyperelliptic involution on $\Sigma_g$).} 

Here is the full version of Theorem~\ref{hyperellipticthm} we quoted in the Introduction:

\begin{theorem}[{\cite[Theorem 1.1]{HS1}}]\label{T:involution HBLF}
Let $f:X\to S^2$ be a hyperelliptic BLF with genus $g\geq 3$. We denote by $C\subset X$ the set of Lefschetz singularities of $f$ with separating vanishing cycles and by $\tilde{X}$ the manifold obtained by blowing up $X$ at the points in $C$. 
\begin{enumerate}
\item 
There exists an involution $\eta$ of $\tilde{X}$ such that $\eta$ is the covering transformation of a double branched covering over the manifold $S\sharp 2\left|C\right|\CPb$, where $S$ is an $S^2$-bundle over $S^2$. 
\item 
The homology class of a regular fiber of $f$ is nontrivial in $H_2(X; \Q)$. 
\end{enumerate}

\end{theorem}

Analogous to the case of honest Lefschetz fibrations, the branched locus consists of exceptional spheres contained among higher side fibers and a simple branched multisection ---but the latter now might have folds along the round locus of the fibration.

Unlike Lefschetz fibrations, the fixed point set of a hyperelliptic BLF often has nonorientable surface components; in fact, in a neighborhood of indefinite folds the fixed point set is the M\"obius band if the corresponding round handle is twisted and is an annulus otherwise, and further, the fixed point set still may be nonorientable even if the round handle is untwisted \cite{HS1}.

\begin{remark}
We shall note that the assumption on the genus of a fibration cannot be dropped from Theorem~\ref{T:involution HBLF}, as one can have twisted gluings of the trivial genus $0$ or $1$ bundle on the lower side to the rest of the fibration. Indeed, there exists a hyperelliptic genus-$1$ BLF on $S^4$ (\cite[Example 1]{ADK}), and neither one of the conclusions of the theorem can hold in this case.
\end{remark}

\begin{remark} \label{hyperNS}
The second conclusion of Theorem~\ref{T:involution HBLF} implies that $X$ is near-symplectic (so $b^+(X)>0$), by the Thurston-Gompf type construction of \cite{ADK}. Another consequence is the non-existence of a genus $g \geq 3$ hyperelliptic BLF on a definite $4$-manifold, such as connected sums of $\CP$s.  
\end{remark}

As in the case of Lefschetz fibrations, one can easily read off several algebraic topological invariants of the total space of a BLF using the simple handlebody decomposition associated to it \cite{B1}. The signature calculation, which in principle is determined by the monodromy factorization, in general is a lot harder than that of say the fundamental group or the Euler characteristic. Nevertheless, extending Endo's work in \cite{Endo}, which greatly simplifies this calculation for hyperelliptic Lefschetz fibrations, in \cite{HS2} the authors showed that there is an easy algorithm for calculating the signature of the total space of a hyperelliptic BLF from a give monodromy factorization \cite{HS2}.

\vspace{0.2cm}
\subsection{Torelli broken Lefschetz fibrations}  \label{TorelliBLFs} \

Let $f: X \to S^2$ be a genus-$g$ BLF with Hurwitz cycle system $W_f=(c; c_1,\ldots, c_n)$. We say that $f$ is a \textit{Torelli broken Lefschetz fibration} if, possibly after conjugating the curves of $W_f$ by a mapping class in $\Map(\Sigma_g)$, we can assume that the monodromy is contained in $\I(\Sigma_g)$. Unlike in the definition of hyperelliptic BLFs, here we make no further assumptions on $c$, as it will be seen shortly that for any Torelli BLF the round cobordism will ``act on the homology trivially'' -- see Remark~\ref{roundTorelli}.\footnote{In analogy to the hyperelliptic case, we could of course ask $t_c$ to be contained in $\I(\Sigma_g)$, which would amount to $c$ being a separating curve, and require allowing simplified BLFs to have separating folds. This however makes almost no difference to the results we present herein; monodromies of such simplified BLFS are also well understood \cite{B1}, we could easily replace the examples in Proposition~\ref{TorelliAbundance} with those having separating round $2$-handles, and Theorem~\ref{almostcomplex} would be proved in an identical way.}

We begin with noting that, in great contrast to genuine Lefschetz fibrations, there exist Torelli BLFs:

\begin{proposition} \label{TorelliAbundance}
There are infinitely many genus-$g$ nontrivial relatively minimal Torelli broken Lefschetz fibrations for every $g \geq 2$. 
\end{proposition}

\begin{proof}
The Figure~\ref{F:TorelliBLF} describes a genus-$g$ broken Lefschetz fibration $f_n: X_n \to S^2$ with isotopic vanishing cycles $c_1, \ldots, c_n$ (iterated along the dotted lines in the figure) bounding a genus one subsurface. (See \cite{B1} for handle diagrams of BLFs.)  Here, the round $2$-handle is made of the $0$-framed $2$-handle on the fiber and a $3$-handle, both given in red. Under the capping homomorphism, the image of the global monodromy of the higher side becomes trivial on the lower side, yielding a true BLF over the $2$-sphere. Observe that the monodromy is supported away from a self-intersection $0$-sphere section. 

The claims of the proposition are now easy to verify: All $c_i$ are separating, so each $f_n$ is Torelli. As $n$ increases by one, so does the Euler characteristic of $X_n$, thus $f_n$ are all distinct.
\begin{figure}[htbp]
\begin{center}
\includegraphics[width=130mm]{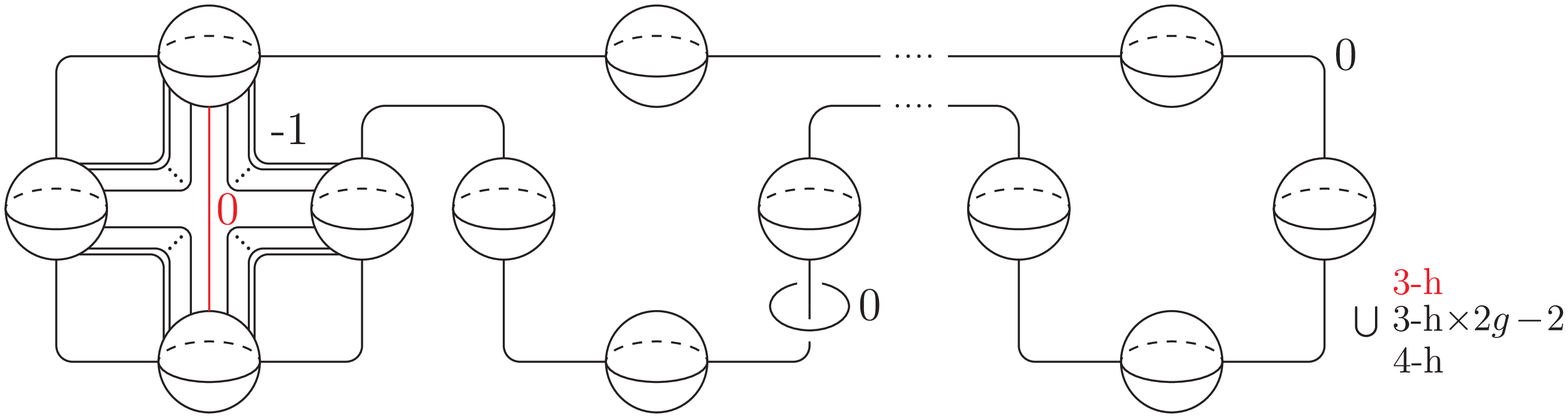}
\end{center}
\caption{Handle diagram of a genus-$g$ Torelli BLF. }
\label{F:TorelliBLF}
\end{figure}
\end{proof}

\noindent 
It is easy to prove by Kirby calculus that the total space of the fibration in \linebreak Figure~\ref{F:TorelliBLF} is $S^1\times S^3 \, \sharp \, S^2\times \Sigma_{g-1} \, \sharp \, n \CPb$. (Hint: slide all Lefschetz $2$-handle over the $0$-handle of the round $2$-handle, and then blow-down the resulting $(-1)$-spheres to arrive at the step fibration \cite{B1}.)

However, the next theorem shows that the Torelli fibrations find room to exist due to lack of almost complex structures on their total spaces:
 
\begin{theorem} \label{almostcomplex}
An almost complex $4$-manifold does not admit a non-trivial relatively minimal Torelli broken Lefschetz fibration. 
\end{theorem}

\begin{proof}
Let $X$ be the total space of a genus-$g$ Torelli BLF. From the corresponding handle decomposition, we calculate the Euler characteristic of $X$ as
\[ \eu(X) = 6-4g+ n . \]

On the other hand, the signature of $X$ is calculated as:
\[ \sigma(X) = -n + \epsilon \, , \]
where $\epsilon \in \{-1, 0, 1\}$. This is because on the higher side, every vanishing cycle is separating and contributes a $-1$ to the signature calculation \cite{Ozbagci}. On the lower side we have a trivial bundle, whose signature is $0$. As shown in \cite[Lemma 4.4]{HS2} the round cobordism contributes either $-1, 0$, or $+1$ to the total signature, and hence our calculation up to determining the exact value of $\epsilon$.

Combining the above we calculate 
\[ \chi_h(X)= \frac{1}{4}(\eu(X) +\sigma(X)) = \frac{1}{4}(6-4g+n - n + \epsilon) = 1-g + \frac{1}{4}(2 + \epsilon) \, ,\]
which can never be an integer. However, $\chi_h$, which would serve as the holomorphic Euler characteristic on an almost-complex $4$-manifold, is not an integer in this case, proving our claim.
\end{proof}

The above theorem, along with Proposition~\ref{TorelliAbundance} completes the proof of Theorem~\ref{torellithm}.

\begin{remark} \label{roundTorelli}
Although we did not need to determine the exact value of $\epsilon$, the signature of a round cobordism, in the proof, it is possible to see that it is $0$ for any round cobordism as above, following \cite{HS2} (to which the reader can refer to for the notations we adopt here).

Since the global monodromy $\varphi$ of the higher side of a genus-$g$ Torelli BLF is in the Torelli group, using the Mayer-Vietoris exact sequence we can obtain an isomorphism $\psi:H_1(V_\varphi;\mathbb{Q})\to H_1(S^1\times \Sigma_g;\mathbb{Q})$ which makes the following diagram commute: 
\[
\begin{CD}
H_1(\Pa V_{\tilde{\varphi}}^\prime;\mathbb{Q}) @> i_\ast >> H_1(V_{\tilde{\varphi}}^\prime;\mathbb{Q}) \\
@V \id_\ast VV @VV \psi V \\
H_1(S^1\times \Pa\Sigma_{g,2};\mathbb{Q}) @> i_\ast >> H_1(S^1\times \Sigma_{g,2};\mathbb{Q}), 
\end{CD}
\]
where $\tilde{\varphi}\in \Map(\Sigma_{g,2})$ is a lift of $\varphi$. 
Chasing this commutative diagram, we conclude that $s(\varphi) = \Sign(m(\varphi)) = \Sign(m(\id_\ast)) = 0$. 
\end{remark}

\begin{remark}
When a given BLF on a $4$-manifold $X$ does not have a special monodromy (say hyperelliptic or Torelli), it is plausible that $X$ admits another BLF which carries the desired property. In fact, it is easy to see that the ``flip-and-slip'' stabilization  \cite{B2} of a hyperelliptic (resp. Torelli) BLF can result in a homotopic BLF which no longer is hyperelliptic (resp. never is Torelli). One can then regard these as $4$-manifold properties instead, and ask how diverse of a picture we get for the special fibrations studied here. 

By Remark~\ref{hyperNS} and Theorem~\ref{almostcomplex}, it is clear that there are many $4$-manifolds which admit neither hyperelliptic nor Torelli BLFs, such as any odd number of connected sums of $\CP$. Torelli BLF examples of Proposition~\ref{TorelliAbundance} can easily be seen to be hyperelliptic as well. In the same construction, one can take non-isotopic vanishing cycles $c_1, \ldots, c_n$ so that they all bound different genus-$1$ subsurfaces containing $c$, and thus we would still have a Torelli BLF. With this freedom of choice in mind, it looks plausible (though we do not attempt to digress more about here) one can choose these $c_i$ in a way that there is no hyperelliptic involution on the fiber for which $c, c_1, \ldots, c_n$ are all symmetric, so as to produce Torelli BLFs which are not hyperelliptic. Lastly, there are many almost complex $4$-manifolds, such as $S^4$, $S^1 \x S^3$ or the minimal rational ruled surfaces which admit genus-$1$ BLFs. So they all admit hyperelliptic BLFs, but no Torelli ones. 
\end{remark}

\vspace{0.2cm}
\subsection{An infinite family of nonisomorphic broken Lefschetz fibrations} \label{nonisomorphicexamples} \

In this section we will construct the examples in Theorem~\ref{infinite}. Let $\Sigma$ be a closed oriented surface of genus $g\geq 2$ and $c\subset\Sigma$ be a non-separating simple closed curve. We denote by $\Sigma_c$ the genus $(g-1)$ surface obtained by surgery along $c$. The surface $\Sigma_c$ has two marked points $p_1, p_2$ at the centers of the disks attached in the surgery. We denote by $\mc{S}_c$ the set of isotopy classes of separating curves in $\Sigma$ which bounds a torus with one boundary component containing $c$, and by $\mc{P}_c$ the set of isotopy classes of simple paths in $\Sigma_c$ between $p_1$ and $p_2$ (where isotopies fix the points $p_1$ and $p_2$). We define a map $\Pi: \mc{S}_c\to \mc{P}_c$ as follows: For $d\in \mc{S}_c$, we take a simple closed curve $\alpha(d)$ in the torus bounded by $d$ so that $\alpha(d)$ intersects $c$ in one point transversely. Then $\Pi(d)$ is defined to be the isotopy class of a path corresponding $\alpha(d)$. For $d_1,d_2\in \mc{S}_c$, we denote by $I(d_1,d_2)$ the geometric intersection number $i(\Pi(d_1),\Pi(d_2))$ between $\Pi(d_1)$ and $\Pi(d_2)$, that is, 
\[
I(d_1,d_2) = \min\{\sharp (\gamma_1\cap\gamma_2)~|~\gamma_i \in\Pi(d_i), \gamma_1 \pitchfork \gamma_2 \}.
\]
We will need the following lemma for our construction:

\begin{lemma}\label{T:invariance I}
Let $f_i:X\to S^2$ ($i=1,2$) be a genus-$g$ BLF and $W_i =(c;d^{(i)}_1,d^{(i)}_2)$ a Hurwitz cycle system of $f_i$.
Suppose that each $d^{(i)}_j$ is contained in $\mc{S}_c$. 
If $f_1$ and $f_2$ are isomorphic, then $I(d_1^{(1)}, d_2^{(1)})$ is equal to $I(d_1^{(2)}, d_2^{(2)})$.
\end{lemma}

\begin{proof}[Proof of Lemma~\ref{T:invariance I}]
Since $f_1$ and $f_2$ are isomorphic, $W_1$ and $W_2$ are Hurwitz equivalent by Theorem~\ref{thm_isom_equiv}. (Note that the ``only if'' part of Theorem~\ref{thm_isom_equiv} still holds without the assumption on genera of fibrations.)
Furthermore, it is obvious that $I(d_1,d_2)$ is invariant under simultaneous conjugations. 
Thus, it is sufficient to prove that $I(d_1,d_2)$ is equal to $I(d_2,t_{d_2}(d_1))$. 

Let $\beta(d_2)$ be the boundary of a regular neighborhood of $\Pi(d_2)$, which is a simple closed curve in $\Sigma_c$. 
This curve corresponds with $d_2\subset \Sigma$ via surgery along $c$.
Since the curve $t_{d_2}(\alpha(d_1))$ is contained in the torus bounded by $t_{d_2}(d_1)$ and intersects $c$ in one point transversely, 
The path $\Pi(t_{d_2}(d_1))$ is isotopic to $t_{\beta(d_2)}(\Pi(d_1))$. 
Thus $I(d_2,t_{d_2}(d_1))$ is calculated as follows:
{\allowdisplaybreaks
\begin{align*}
I(d_2,t_{d_1}(d_2)) & = i(\Pi(d_2),\Pi(t_{d_2}(d_1))) \\
& = i(t_{\beta(d_2)}(\Pi(d_2)),t_{\beta(d_2)}(\Pi(d_1))) \\
& = i(\Pi(d_2),\Pi(d_1)) = I(d_1,d_2).
\end{align*}
}
\end{proof}

\begin{proof}[Proof of Theorem~\ref{infinite}]
We take curves $d,d_1,d_2^0\in \mc{S}_c$ as shown in Figure~\ref{F:separating scc path} and denote $t_{d}^n(d_2^0)$ by $d_2^n$.
\begin{figure}[htbp]
\centering
\subfigure[]{\includegraphics[width=90mm]{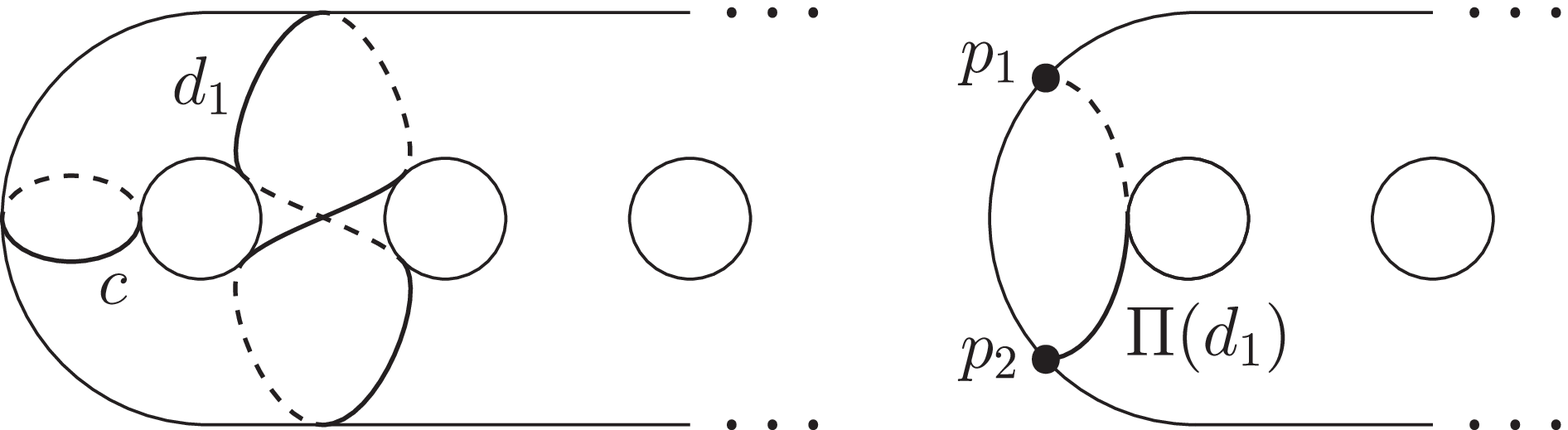}
\label{F:separating scc path1}}
\subfigure[]{\includegraphics[width=90mm]{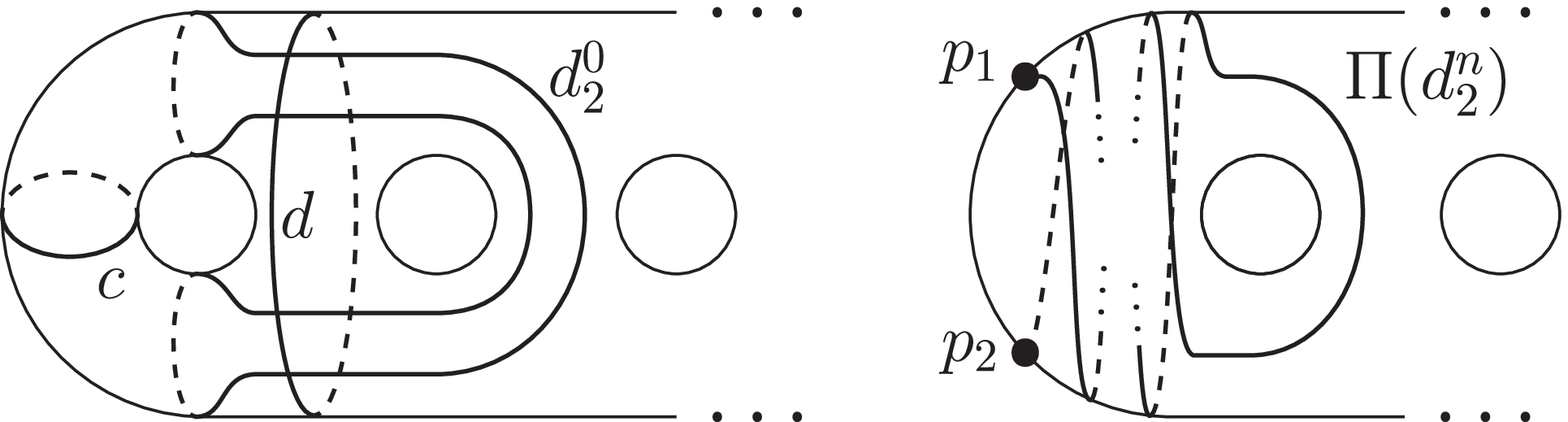}
\label{F:separating scc path2}}
\caption{Separating curves and corresponding paths.}
\label{F:separating scc path}
\end{figure}
The product $t_{d_2^n}t_{d_1}$ is contained in $\ker(\Phi_c)$. 
Thus there exists a genus-$g$ BLF $f_n:X_n\to S^2$ whose Hurwitz cycle system is $(c;d_1,d_2^n)$.
Using the bigon criterion for paths \cite[Section 1.2.7]{FM}, we can prove that $I(d_1,d_2^n)=4n+2$ for any $n\geq 0$. In particular $f_n$ is not isomorphic to $f_m$ for distinct $n,m\geq 0$.

Since both of the curves $d_1$ and $d_2^n$ are disjoint from $c$, we can change $f_n$ by a \textit{homotopy} preserving tubular neighborhoods of regular fibers (in both higher and lower sides) so that the two Lefschetz singularities are in the lower side of $f_n$. The Lefschetz vanishing cycles of the resulting fibration $\hat{f}_n$ are null-homotopic since both $d_1$ and $d_2^n$ bound tori containing $c$. Thus $\hat{f}_n$ and $\hat{f}_m$ are isomorphic, in particular there is a diffeomorphism $\Phi_{n,m}:X_n \to X_m$ for any $n,m\geq 0$. Furthermore $\Phi_{n,m}$ sends tubular neighborhood of a regular fiber of $f_n$ (preserved by the homotopy above) to that of $f_m$. Hence $\{h_n:=f_n\circ \Phi_{0,n} \}_{n\geq 0}$ is an infinite family of non-isomorphic BLFs on $X=X_0$ such that each fibration in the family has the same surfaces as regular fibers. 
\end{proof}

\begin{remark} \label{cohomotopy}
It is well-known that the cohomotopy set $\pi^2(X)= [X, S^2]$ is isomorphic to $\mathbb{F}_2(X)$, the set of framed surfaces in $X$ (e.g. \cite{KMT}). Our theorem thus shows that there is an infinite family of BLFs $(X, h_n)$, $n \in \N$, within the same cohomotopy class, demonstrating the rigidity of BLF isomorphisms versus homotopy equivalence of indefinite generic maps.
\end{remark}

\vspace{0.2in} \noindent \textit{Acknowledgements.} The first author was partially supported by Simons Foundation Grant $317732$. The second author was supported by JSPS Grant-in-Aid for Young Scientists (B) $26800027$.

\vspace{0.4in}

\end{document}